\documentclass{amsart}
\usepackage{tikz}

\usepackage{amsmath, amsthm, xypic, amssymb, enumerate, mathrsfs, eufrak, verbatim, graphicx, bbm, stackrel,MnSymbol}
\usepackage{color}
\usepackage[toc,page]{appendix}
\usepackage[mathscr]{euscript}
\usepackage[all,cmtip,2cell]{xy}
\xyoption{rotate}
\UseTwocells

\usepackage[all]{xy}
\usepackage{hyperref}
\usepackage{tikz}
\usepackage{tikz-3dplot}

\DeclareMathOperator{\col}{colim}
\DeclareMathOperator*{\hc}{hocolim}

\DeclareMathOperator{\Hom}{Hom}

\newcommand{\cC}{\mathcal{C}}
\newcommand{\sD}{\mathcal{D}}

\newcommand{\cO}{\mathcal{O}}

\newcommand{\cX}{\mathcal{X}}

\newcommand{\NN}{\mathbb{N}}

\def\mapdown#1.{\Big\downarrow\rlap{$\vcenter{\hbox{$\scriptstyle#1$}}$}}

\newtheorem{thm}{Theorem}[section]

\newtheorem*{thm*}{Theorem}
\newtheorem{prop}[thm]{Proposition}

\newtheorem{lemma}[thm]{Lemma}

\theoremstyle{definition}
\newtheorem{defn}[thm]{Definition}

\theoremstyle{definition}
\newtheorem{ex}[thm]{Example}

\theoremstyle{definition}
\newtheorem{rem}[thm]{Remark}

\theoremstyle{definition}

\theoremstyle{definition}

\newtheorem*{rem*}{Remark}
\newtheorem{exA}{Example}

\newtheoremstyle{TheoremForIntro} 
        {.6em}{.6em}              
        {\itshape}                      
        {}                              
        {\bfseries}                     
        {. }                             
        { }                             
        {\thmname{#1}\thmnote{ \bfseries #3}}
    \theoremstyle{TheoremForIntro}
    \newtheorem{TheoremIntro}[thm]{Theorem}

\input xy
\xyoption{all}

\newcommand{\itop}{\mathsf{isvt\text{-}Top}}
\newcommand{\etop}{\mathsf{eqvt\text{-}Top}}
\newcommand{\topp}{\mathsf{Top}}
\newcommand{\strat}{\mathsf{strat}}
\newcommand{\imap}{\mathsf{Map_{isvt}}}
\newcommand{\emap}{\mathsf{Map_{eqvt}}}
\newcommand{\map}{\mathsf{Map}}
\newcommand{\fun}{\mathsf{Fun}}

\newcommand{\term}{\triangleright}

\title{An isovariant Elmendorf's theorem}
\author{Sarah Yeakel}

\begin{document}

\begin{abstract}
An isovariant map between spaces with a group action is an equivariant map which preserves isotropy groups. In this paper, we show that for a finite group $G$, the category of $G$-spaces with isovariant maps has a Quillen model structure.  We prove a Piacenza-style model theoretic proof of an isovariant Elmendorf's theorem, showing that this model structure is Quillen equivalent to a model category of diagrams.
\end{abstract}

\maketitle
\thispagestyle{empty}


\section*{Introduction}

Let $G$ be a finite group and let $X$ and $Y$ be compactly generated spaces with continuous left $G$-actions. 
	A map $f:X \to Y$ is \emph{equivariant} if it preserves the $G$-action, so $g \cdot f(x)=f(g \cdot x)$. One consequence of equivariance is that $f$ can increase isotropy subgroups, that is, $G_x \subseteq G_{f(x)}$, where $G_x=\{g \in G: g \cdot x=x\}$. A map $f:X \to Y$ is \emph{isovariant} if it is equivariant and $G_x=G_{f(x)}$ for all $x \in X$.

\begin{exA} \label{exD}
	Let $C_2$ be the cyclic group with two elements. Consider the one-point space $\ast$ with trivial $C_2$-action and the unit disk $D^2=\{(x,y) \in \mathbb{R}^2: x^2 + y^2 \leq 1\}$ with the $C_2$-action which reflects across the $y$-axis. Any map from $\ast$ to $D^2$ whose image is on the $y$-axis is both equivariant and isovariant. In fact, any injective equivariant map is isovariant. On the other hand, the map $D^2 \to \ast$ is equivariant, but not isovariant.
\end{exA} 

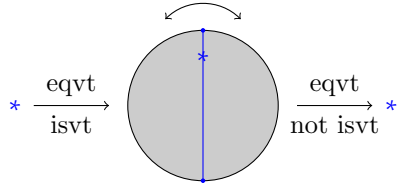
\begin{center}
\begin{tikzpicture}[scale=.45]
	\node[blue] at (-5,0) {$\ast$};
	\draw[->] (-4.5,0)--(-2.5,0);
	\node[above] at (-3.5,0) {eqvt};
	\node[below] at (-3.5,0) {isvt};
	\draw[<->] (-1,2.3) to [out=45,in=135] (1,2.3);
	\filldraw[fill=black!20] (0,0) circle (2cm);
	\filldraw[blue] (0,2) circle (1.3pt);
	\filldraw[blue] (0,-2) circle (1.3pt);
	\node[blue] at (0,1.3) {$\ast$};
	\draw[blue] (0,2)--(0,-2);
	\draw[->] (2.5,0)--(4.5,0);
	\node[above] at (3.5,0) {eqvt};
	\node[below] at (3.5,0) {not isvt};
	\node[blue] at (5,0) {$\ast$};
\end{tikzpicture}
\end{center}

Equivariant homotopy theory has led to important computations in the nonequivariant setting, and isovariant maps occur naturally in the study of surgery theory and classification questions for manifolds \cite{browderquinn}. The usual approach in the isovariant setting is to require extra assumptions on the dimensions of gaps between isotropy subspaces under which isovariant and equivariant homotopy equivalences coincide \cite{dulaschultz,schultzgap}. The further development of isovariant homotopy theory could provide an alternative to gap hypothesis assumptions. 
In this paper, we prove isovariant analogues of equivariant categorical foundations for finite groups with the expectation that the extra homotopical structure will yield new computational techniques. We motivate the definitions of the isovariant orbit category $\mathcal{L}_G$ and linking simplices $\Delta^\bullet_G$ by defining isovariant cell complexes. Assuming familiarity with cofibrantly generated model structures \cite{hovey}, we define a Quillen model structure on the category of $G$-spaces with isovariant maps $\itop$. Introducing a formal terminal object allows us to show that $\itop^\term$ is Quillen equivalent to presheaves on the isovariant orbit category. Our main result is the following isovariant analogue of Elmendorf's theorem.

\begin{TheoremIntro}[\ref{mainthm}]
For a finite group $G$, the following adjunction is a Quillen equivalence.
	\[
\Delta^\bullet_G \otimes_{\mathcal{L}_G} -: \xymatrix{ \mathsf{Fun}(\mathcal{L}^{op}_G, \mathsf{Top}) \ar@<1ex>[r] & \itop^\term \ar@<1ex>[l]}: \imap(\Delta^\bullet_G,-)
	\]
\end{TheoremIntro}

\subsection*{Equivariant precursors}

Let $\topp$ denote the convenient category of compactly generated spaces, which are not necessarily weak Hausdorff (see Remark \ref{referee}). 
While much of equivariant homotopy theory has been developed for compact Lie groups, in this paper we work only with finite groups, and we fix a finite group $G$.
Let $\etop$ denote the category of $G$-spaces with equivariant maps, and let $\itop$ denote the category of $G$-spaces with isovariant maps. 
The category $\etop$ is enriched in spaces, and thus $\itop$ is also enriched in spaces using the subspace topology. We will denote the space of isovariant maps from $X$ to $Y$ by $\imap(X,Y)$.

An important tool for working with $\etop$ is the orbit category for the group $G$. The $G$-orbit category $\mathcal{O}_G$ has as objects the $G$-sets $G/H$ for all subgroups $H\leq G$, and morphisms $G/H \to G/K$ are equivariant maps.

By adjunction, there is a natural homeomorphism $\emap(G/H, G/K)\cong (G/K)^H$. An equivariant map $G/H \to G/K$ is completely determined by its value on the identity coset, $eH \mapsto \gamma K$. Thus a morphism $R_\gamma:G/H \to G/K$ in $\mathcal{O}_G$ exists if and only if $H$ is subconjugate to $K$, that is, $\gamma H\gamma^{-1} \leq K$. For $H=K$, such elements $\gamma \in G$ constitute the Weyl group $\gamma \in N_G(H)/H$.

There is a cofibrantly generated model structure on the category $\topp$ of spaces where the generating cofibrations are given by $\{S^{n-1} \to D^n\}$ and generating acyclic cofibrations are given by $\{D^n \times \{0\} \to D^n \times [0,1]\}$ \cite[2.4]{hovey}. The family of adjunctions below allows this model structure to transfer to $\etop$ \cite{stephanelm}. 
\[
\left\{G/H \times -:\xymatrix{\mathsf{Top} \ar@<1ex>[r] & \etop \ar@<1ex>[l] } : \emap(G/H, -) \right \}_{G/H \in \mathcal{O}_G}
\]

In \cite{elmthm}, Elmendorf explicitly shows that the equivariant homotopy theory of $G$-spaces (for $G$ a compact Lie group) agrees with the homotopy theory of diagrams indexed on the orbit category. In \cite{alaskanotes}, the result is extended to any topological group, and in \cite{piacenza}, the theorem is reformulated to a Quillen equivalence of model categories.

\begin{thm*}
	[\cite{piacenza}] The following adjunction is a Quillen equivalence.
	\[ eval_{G/e}:
\xymatrix{ \mathsf{Fun}(\mathcal{O}^{op}_G, \mathsf{Top}) \ar@<1ex>[r] & \etop \ar@<1ex>[l]}: \emap(\bullet, -)
	\]
\end{thm*}

The right adjoint is the functor taking a $G$-space $X$ to the diagram $X^{(-)}:\mathcal{O}_G^{op} \to \mathsf{Top}$ which sends $G/H \mapsto X^H\cong\emap(G/H,X)$. The left adjoint is evaluation of the diagram at the orbit $G/e$. The projective model structure on the category of $\mathcal{O}_G$-diagrams is briefly described in section \ref{elmsection}, but more details can be found in \cite{piacenza}.

In the category $\itop$ of $G$-spaces with isovariant maps, the morphisms induce maps on the stratification of the spaces by isotropy group. To capture this extra structure, the replacement of the orbit category will have as objects chains of subgroups of $G$. We call this category the \emph{link orbit category}, $\mathcal{L}_G$ (see Definition 1.1). Since isovariant maps preserve stratification by isotropy, the morphisms in the link orbit category preserve chains of subgroups and induce equivariant maps on the corresponding orbits. 

For each chain $H_0< \cdots< H_n$ of subgroups of the finite group $G$, we define a \emph{linking simplex} $\Delta_G^{H_1< \cdots < H_n}$ as a quotient of $G \times \Delta^n$ manifesting the subgroups $H_i$ as isotropy groups. We show $\Delta_G^\bullet$ defines a functor from the link orbit category to $\itop$. The linking simplices fit into a family of adjunctions between the categories of spaces $\topp$ and isovariant $G$-spaces $\itop$. 
\[
\{\Delta_G^{H_0< \cdots < H_n} \times -: \xymatrix{\mathsf{Top} \ar@<1ex>[r] & \itop \ar@<1ex>[l] } : \imap(\Delta^{H_0< \cdots < H_n}_G, -)\}_{H_0< \cdots < H_n \in \mathcal{L}_G}
\]

We use Kan's transfer principle to transfer the cofibrantly generated model structure on $\topp$ across the adjunctions, which endows $\itop$ with a model structure, but to be a model category requires $\itop$ to have all small limits and colimits. 
While the category $\itop$ does not have a terminal object, we may assign it a formal terminal object, completing the category to a model category $\itop^\term$. A homotopy terminal object for $\itop$ can be constructed as the image under the derived left adjoint $\Delta^\bullet_G \otimes_{\mathcal{L}_G} -$ of the constant diagram on a point, and we assume our terminal object has the same isovariant homotopy type.

Finally, using the projective model structure on the category of diagrams $\mathsf{Fun}(\mathcal{L}_G^{op}, \mathsf{Top})$ and the transferred model structure on $\itop$ extended to $\itop^\term$, we prove the isovariant analogue of Elmendorf's theorem. 
A nonformal ingredient of the proof is Lemma \ref{linkscellular}, which states that the linking simplex functors $\imap(\Delta_G^\bullet,-)$ are cellular in the sense that they preserve sequential homotopy colimits and homotopy pushouts up to homotopy. We prove this using semicontinuity arguments and careful deformations of simplices in the cylinder coordinates of the homotopy colimits.

The paper is organized as follows. In section \ref{links}, we define the link orbit category $\mathcal{L}_G$ and the linking simplices $\Delta^\bullet_G$. 
In section \ref{modelstr}, we prove that $\itop$ has a cofibrantly generated model structure transferred across adjunctions with $\topp$ and extend this structure to $\itop^\term$. Finally, in section \ref{elmsection}, we prove the main theorem, an isovariant version of Elmendorf's theorem.

\vspace{.15in}

\emph{Acknowledgements.} This project was inspired by conversations at a Women in Topology workshop at the Mathematical Sciences Research Institute. I would like to thank Kate Ponto, Cary Malkiewich, and Inbar Klang for numerous illuminating discussions on isovariance. 
I would also like to thank the referee whose detailed report significantly improved this paper.

\vspace{.15in}

\section{The link orbit category and linking simplices} \label{links}

We motivate the definition of the link orbit category for a finite group $G$ by considering how to modify $G$-CW complexes to the isovariant category. 

In the standard model structure on $\etop$, the cofibrant objects are retracts of $G$-CW complexes, which are built by attaching entire orbits at a time. That is, a cell is $G/H \times D^n$ with a trivial action on $D^n$, and the cell is equivariantly glued along its boundary orbit $G/H \times S^{n-1}$. For example, as a $C_2$-CW complex, the flip disk $D^2$ in Example \ref{exD} is built from two fixed 0-cells $G/G \times D^0$ denoted $x,y$, one fixed 1-cell $G/G \times D^1$ denoted $z$ glued from $y$ to $x$, one free 1-cell $G/e \times D^1$ denoted $w$ glued from $x$ to $y$, and one free 2-cell $G/e \times D^2$ glued along its boundary to $w$ followed by $z$. 

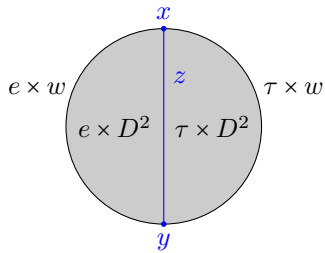
\begin{center}
\begin{tikzpicture}[scale=.65]
	\filldraw[fill=black!20] (0,0) circle (2cm);
	\filldraw[blue] (0,2) circle (1.3pt);
	\node[above, blue] at (0,2) {$x$};
	\node at (-1,0) {$e \times D^2$};
	\node[left] at (-1.8,.8) {$e \times w$};
	\node[right] at (1.85,.8) {$\tau \times w$};
	\node at (1,0) {$\tau \times D^2$};
	\node[right, blue] at (0,1) {$z$};
	\filldraw[blue] (0,-2) circle (1.3pt);
	\node[below, blue] at (0,-2) {$y$};
	\draw[blue] (0,2)--(0,-2);
\end{tikzpicture}
\end{center}

The construction of this $C_2$-CW complex does not transfer to the category $\itop$, because the attaching maps are not isovariant; for example, attaching the free 1-cell $w$ to the fixed points uses a non-isovariant map $G/e \times S^0 \to \{x,y\}$. Building isovariant cell complexes will involve introducing new cells which interpolate between the different strata of fixed points. 

We define \emph{linking simplices} $\Delta^{H_1< \cdots <H_n}_G$ that play the role of orbits in the isovariant category. Then the \emph{linking cells} will be $\Delta^{H_1< \cdots <H_n}_G \times D^n$, in analogy with the cells $G/H \times D^n$ in the equivariant category. We use the word link because the 1-dimensional building block was called a link in Quinn's work on homotopically stratified spaces \cite{quinn}. The linking simplices can also be described in terms of the nondegenerate simplices of the \emph{exit path category} of Lurie \cite[A.6]{lurieHA}.


Let $\Delta^n$ be the standard $n$-simplex in $\mathsf{Top}$, that is, \[\Delta^n=\left\{ (t_0, \dots, t_n) \in [0,1]^{n+1} : \sum_{i=0}^n t_i = 1\right\}.\] 

We define the 1-dimensional linking simplex to be $\Delta^{e<G}= G \times \Delta^1  / G \times \{(1,0)\}$. The case $G=C_2$ is pictured.

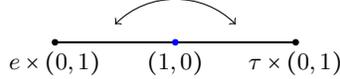
\begin{center}
\begin{tikzpicture}[scale=.8]
	\draw[thick] (-2,0)--(2,0);
	\filldraw[blue] (0,0) circle (1.3pt);
	 \node[below] at (0,0) {\small $(1,0)$};
	 \node [below] at (-2,0) {\small $e \times (0,1)$};
	 \node[below] at (2,0) {\small $\tau \times (0,1)$};
	\draw[<->] (-1,.3) to [out=45,in=135] (1,.3);
	\filldraw (-2,0) circle (1.3pt);
	\filldraw (2,0) circle (1.3pt);
\end{tikzpicture}
\end{center}

For subgroups $H\leq G$, the orbit $G/H$ can be interpreted as a 0-dimensional link simplex $\Delta^H=G/H \times \Delta^0$. Then the flip disk $D^2$ can be built as an isovariant cell complex with 
two link 0-cells $\Delta^{e<G} \times D^0$ (labeled $m_1,m_2$) and one link 1-cell $\Delta^{e<G} \times D^1$ glued to $m_1, m_2$ as pictured.

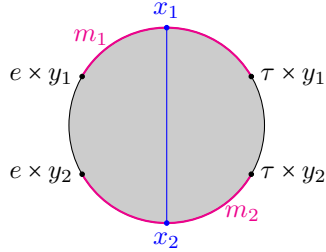
\begin{center}
\begin{tikzpicture}[scale=.65]
	\filldraw[fill=black!20] (0,0) circle (2cm);
	\draw[magenta, thick] (1.73,1) arc (30:150:2cm);
	\draw[magenta, thick] (-1.73,-1) arc (210:330:2cm);
	\filldraw[blue] (0,2) circle (1.3pt);
	\filldraw[black] (-1.73,1) circle (1.3pt);
	\filldraw[black] (1.73,1) circle (1.3pt);
	\filldraw[black] (-1.73,-1) circle (1.3pt);
	\filldraw[black] (1.73,-1) circle (1.3pt);
	\node[left,magenta] at (-1,1.8) {$m_1$};
	\node[right,magenta] at (1,-1.8) {$m_2$};
	\filldraw[blue] (0,-2) circle (1.3pt);
	\draw[blue] (0,2)--(0,-2);
\end{tikzpicture}
\end{center}

For finite groups with longer chains of subgroups, we use higher-dimensional link simplices to keep track of how all subgroups in a chain interact with each other. For example, for $e<H<G$, the fundamental domain of the space $\Delta^{e<H<G}$ is $\Delta^2$ with $(1,0,0)$ fixed by $G$, the edge from $(1,0,0)$ to $(0,1,0)$ fixed by $H$, and the rest of the simplex fixed only by $e$.

\begin{center}
\begin{tikzpicture}scale=.7]
	\filldraw[fill=black!20]  (1,1) -- (0,0) -- (2,0)--(1,1);
	\draw[magenta,thick] (0,0) -- (2,0);
	\filldraw (1,1) circle (1.5pt) node[above]{\small $(0,0,1)$}; 
	\node[above right] at (1.5,.5) {fixed by $e$};
	\filldraw[magenta] (2,0) circle (1.5pt) node[below left]{fixed by $H$};
	\filldraw[blue] (0,0) circle (1.5pt)node [below left]{fixed by $G$};
	\node[above left] at (0,0) {\small $(1,0,0)$};
	\node[above right] at (2,0) {\small $(0,1,0)$};
\end{tikzpicture}
\end{center}

The linking simplices are naturally indexed on the link orbit category.

\begin{defn}
	The \emph{link orbit category}, $\mathcal{L}_G$, can be constructed in two stages. Let $[n]=\{0, 1, \dots ,n\}$, and let $\widetilde{\mathcal{L}}_G$ be the small category with objects given by strictly increasing chains of subgroups of $G$, $\mathbf{H}=H_0 < \cdots < H_n$. Let the morphisms $\widetilde{\mathcal{L}}_G(H_0< \cdots< H_n, K_0 < \cdots <K_m)$ be pairs $(\iota, \gamma)$, where $\iota:[n] \to [m]$ is an order-preserving inclusion and $\gamma \in G$ is an element such that $H_j=\gamma K_{\iota(j)} \gamma^{-1}$ for all $j \in [n]$. That is, $\gamma$ induces a map $R_\gamma: G/H_j \to G/K_{\iota(j)}: gH_j \mapsto g\gamma K_{\iota(j)}$ for all $j$. Composition in $\widetilde{\mathcal{L}}_G$ is defined by $(\iota, \gamma) \circ (\iota', \gamma')=(\iota \circ \iota', \gamma' \gamma)$. Then the link orbit category $\mathcal{L}_G$ is the quotient of $\widetilde{\mathcal{L}}_G$ in which two parallel morphisms $(\iota, \gamma),(\iota', \gamma'):\mathbf{H} \to \mathbf{K}$ are identified if $\iota=\iota'$ and $R_\gamma=R_{\gamma'}:G/H_j \to G/K_{\iota(j)}=G/K_{\iota'(j)}$ for all $j$. We denote the morphisms in $\mathcal{L}_G$ by $(\iota, \overline{\gamma})$, and note that $\overline{\gamma} \in \cap_{j\in [n]} C(H_j,K_{\iota(j)})/K_{\iota(0)}$, where $C(H_j, K_{\iota(j)})=\{\gamma \in G : \gamma^{-1}H_j\gamma=K_{\iota(j)}\}$. That is, $\overline{\gamma}$ is a class of elements of $G$ which simultaneously conjugate $H_j$ to $K_{\iota(j)}$ for all $j \in [n]$.
\end{defn}



While the link orbit category $\mathcal{L}_G$ contains the objects of the ordinary orbit category $\cO_G$ as the length zero chains, $\mathcal{L}_G$ is missing all non-self-maps between them. For example, for $G=C_2$, the orbit category $\mathcal{O}_{C_2}$ has a map $G/e \to G/G$, but the link orbit category $\mathcal{L}_{C_2}$ does not have any map $e \to G$ because $e$ is not a subchain of $G$. This is manifesting the fact that there is no isovariant map $\Delta^e \to \Delta^G$.

\begin{defn}
The \emph{linking simplex} $\Delta^{\mathbf{H}}_G$ for a chain $\mathbf{H}=H_0< \cdots < H_n$ is the quotient  
\[\Delta^{H_0< \cdots < H_n}_G=(G \times \Delta^n)/\sim \] where $(g,x)\sim(g',x)$ if and only if $gH_{k}=g'H_{k}$, when $x=(t_0, \dots, t_{n-k}, 0, \dots, 0)$, $0 \leq k \leq n$. Let $G \times \Delta^n \to \Delta^{H_0< \cdots< H_n}_G$ be the natural projection and denote the image of $(g,x)\in G \times \Delta^n$ under the projection by $\langle g,x\rangle$. The space $\Delta^{\mathbf{H}}_G$ has a left $G$-action given by $g' \cdot \langle g,x \rangle=\langle g'g,x \rangle$; points of the form $\langle g,(t_0, \dots, t_{n-k},0, \dots,0)\rangle$ where $t_{n-k}\neq 0$ are fixed by $gH_kg^{-1}$ under the $G$-action. 
\end{defn}

We note that $\Delta^{H_0< \cdots < H_k}_G$ is the same as the ``equivariant simplex'' $\Delta_k(G;H_k, \dots, H_0)$ defined in \cite{illmanCW}, although in the equivariant simplex, subgroups may be repeated. Illman shows that the equivariant simplex is a compact Hausdorff space with orbit space $\Delta^k$.

For the length 0 chain objects $H_0$ of $\mathcal{L}_G$, $\Delta^{H_0}_G \cong G/H_0$. The fundamental domain of the $G$-space $\Delta^{e<G}_G = G \times \Delta^1/G \times \{(1,0)\}$ is homeomorphic to an interval with 0 fixed. The category $\mathcal{L}_G$ captures higher dimensional links for the stratification by isotropy, and this extends to a functor $\Delta^\bullet_G:\mathcal{L}_G \to \itop$. We describe the maps $\Delta^{(\iota,\overline{\gamma})}$ now.

For an order-preserving inclusion $\iota:[n]\to [m]$, we define a map $\iota_\ast:\Delta^n \to \Delta^m$ which sends $x=(t_0, \dots, t_n)$ to the $m+1$-tuple $(s_0, \dots, s_m)$ defined by 
\[
s_{m-j}=\sum_{\iota(k)=j} t_{n-k}.
\]

\begin{rem}
	This is not the usual induced map on standard simplices. For example, for the ordered inclusion $\iota: \{0,1,2\}\mapsto\{\hat{0}, 1, 2, \hat{3}, 4 \}$, the map $\iota_\ast$ is defined by $\iota_\ast(t_0, t_1, t_2)=(t_0,0,t_1, t_2, 0)$.
\end{rem} 

Note that $\iota_\ast$ sends the vertex $(0, \dots, 0, 1, 0, \dots, 0)$ with a 1 in the $n-k$th spot to the vertex with a 1 in the $m-\iota(k)$th spot. Similarly, $\iota_\ast(t_0, \dots, t_{n-k}, 0, \dots, 0)$ has zeros above the $m-\iota(k)$th spot.


Then we define the map $\Delta^{(\iota,\overline{\gamma})}:\Delta^{\mathbf{H}}\to \Delta^{\mathbf{K}}$ which takes $\langle g,x \rangle \in \Delta^\mathbf{H}$ to $\langle g\gamma, \iota_\ast(x)\rangle \in \Delta^{\mathbf{K}}$. If $(g,x) \sim (g',x)$, then $gH_k=g'H_k$ and $x=(t_0, \dots, t_{n-k}, 0, \dots, 0)$, thus $\iota_\ast(x)$ has zeros above the $m-\iota(k)$ spot. For any representative $\gamma$ of $\overline{\gamma}$, we note that $g\gamma K_{\iota(k)}=g H_k\gamma=g' H_k\gamma=g'\gamma K_{\iota(k)}$, so the map $\Delta^{(\iota, \overline{\gamma})}$ is well-defined. Since $\Delta^{(\iota, \overline{\gamma})}$ commutes with the $G$-action and is injective, it is also isovariant. It is not hard to check the following.

\begin{prop}
	The linking simplices define a functor $\Delta_G^\bullet:\mathcal{L}_G \to \itop$.
\end{prop}


\begin{ex}
	The linking simplices $\Delta^{\bullet}_{C_4}$ and their (non-self) maps are pictured below.

	\vspace{.1in}

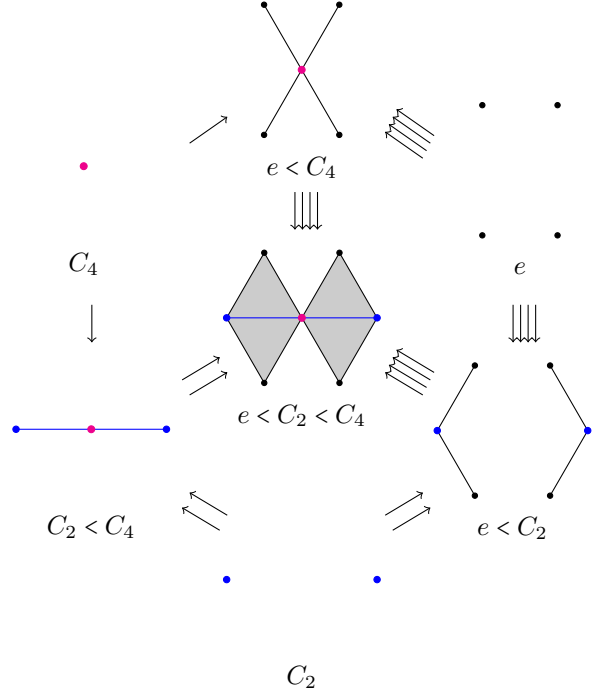
	\newsavebox{\mybox}
\sbox{\mybox}{%
   \begin{tikzpicture}
     \filldraw[fill=black!20] (-.5,-.866) -- (.5,.866)--(1,0) --(.5,-.866)--(-.5,.866)--(-1,0)--(-.5,-.866);
		\draw[blue] (-1,0)--(1,0);
		\filldraw[magenta] (0,0) circle (1.3pt);
		\filldraw[blue] (-1,0) circle (1.2pt);
		\filldraw[blue] (1,0) circle (1.2pt);
		\filldraw (.5,.866) circle (1pt);
		\filldraw (-.5,.866) circle (1pt);
		\filldraw (-.5,-.866) circle (1pt);
		\filldraw (.5,-.866) circle (1pt);
		\node at (0,-1.3) {$e<C_2<C_4$};
   \end{tikzpicture}
 }

 \newsavebox{\myboxb}
\sbox{\myboxb}{%
   \begin{tikzpicture}
     \draw (-.5,-.866) -- (.5,.866);
		\draw (.5,-.866)--(-.5,.866);
		\filldraw[magenta] (0,0) circle (1.3pt);
		\filldraw (.5,.866) circle (1pt);
		\filldraw (-.5,.866) circle (1pt);
		\filldraw (-.5,-.866) circle (1pt);
		\filldraw (.5,-.866) circle (1pt);
		\node at (0,-1.3) {$e<C_4$};
   \end{tikzpicture}
 }
 \newsavebox{\myboxc}
\sbox{\myboxc}{%
   \begin{tikzpicture}
      \draw (.5,.866)--(1,0) --(.5,-.866);
		\draw (-.5,.866)--(-1,0)--(-.5,-.866);
		\filldraw[blue] (-1,0) circle (1.2pt);
		\filldraw[blue] (1,0) circle (1.2pt);
		\filldraw (.5,.866) circle (1pt);
		\filldraw (-.5,.866) circle (1pt);
		\filldraw (-.5,-.866) circle (1pt);
		\filldraw (.5,-.866) circle (1pt);
		\node at (0,-1.3) {$e<C_2$};
   \end{tikzpicture}
 }
 \newsavebox{\myboxd}
\sbox{\myboxd}{%
   \begin{tikzpicture}
      \draw[white] (-.5,-.866) -- (.5,.866)--(1,0) --(.5,-.866)--(-.5,.866)--(-1,0)--(-.5,-.866);
		\draw[blue] (-1,0)--(1,0);
		\filldraw[magenta] (0,0) circle (1.3pt);
		\filldraw[blue] (-1,0) circle (1.2pt);
		\filldraw[blue] (1,0) circle (1.2pt);
		\node at (0,-1.3) {$C_2<C_4$};
   \end{tikzpicture}
 }
  \newsavebox{\myboxe}
\sbox{\myboxe}{%
   \begin{tikzpicture}
      \filldraw (.5,.866) circle (1pt);
		\filldraw (-.5,.866) circle (1pt);
		\filldraw (-.5,-.866) circle (1pt);
		\filldraw (.5,-.866) circle (1pt);
		\node at (0,-1.3) {$e$};
   \end{tikzpicture}
 }

  \newsavebox{\myboxf}
\sbox{\myboxf}{%
   \begin{tikzpicture}
      \draw[white] (-.5,-.866) -- (.5,.866)--(1,0) --(.5,-.866)--(-.5,.866)--(-1,0)--(-.5,-.866);
		\filldraw[blue] (-1,0) circle (1.2pt);
		\filldraw[blue] (1,0) circle (1.2pt);
		\node at (0,-1.3) {$C_2$};
   \end{tikzpicture}
 }
  \newsavebox{\myboxg}
\sbox{\myboxg}{%
   \begin{tikzpicture}
      \draw[white] (-.5,-.866) -- (.5,.866)--(1,0) --(.5,-.866)--(-.5,.866)--(-1,0)--(-.5,-.866);
		\filldraw[magenta] (0,0) circle (1.3pt);
		\node at (0,-1.3) {$C_4$};
   \end{tikzpicture}
 }

\begin{center}
\begin{tikzpicture}
  \node at (0,0) {\usebox{\mybox}};
  \node at (0,3.3) {\usebox{\myboxb}};
  \draw[->] (-.05,2)--(-.05,1.5);
  \draw[->] (.15,2)--(.15,1.5);
  \draw[->] (.05,2)--(.05,1.5);
  \draw[->] (-.15,2)--(-.15,1.5);
  \node at (2.8,-1.5) {\usebox{\myboxc}};
  \draw[->] (1.6,-.6)--(1.1,-.3);
  \draw[->] (1.55,-.7)--(1.05,-.4);
  \draw[->] (1.65,-.5)--(1.15,-.2);
  \draw[->] (1.7,-.4)--(1.2,-.1);
  \node at (-2.8,-1.5) {\usebox{\myboxd}};
    \draw[->] (-1.55,-.7)--(-1.05,-.4);
  \draw[->] (-1.65,-.5)--(-1.15,-.2);
  \node at (2.9,2) {\usebox{\myboxe}};
  \draw[->] (1.6,2.55)--(1.1,2.9);
  \draw[->] (1.55,2.45)--(1.05,2.8);
  \draw[->] (1.65,2.65)--(1.15,3);
  \draw[->] (1.7,2.75)--(1.2,3.1);
  \draw[->] (2.85,.5)--(2.85,0);
  \draw[->] (3.05,.5)--(3.05,0);
  \draw[->] (2.95,.5)--(2.95,0);
  \draw[->] (2.75,.5)--(2.75,0);
  \node at (-2.9,2) {\usebox{\myboxg}};
  \draw[->] (-1.55,2.65)--(-1.05,3);
  \draw[->] (-2.85,.5)--(-2.85,0);
  \node at (0,-3.5) {\usebox{\myboxf}};
  \draw[->] (-1.05,-2.3)--(-1.55,-2);
  \draw[->] (-1.15,-2.5)--(-1.65,-2.2);
  \draw[->] (1.05,-2.3)--(1.55,-2);
  \draw[->] (1.15,-2.5)--(1.65,-2.2);
\end{tikzpicture}
\end{center}
\end{ex}



\vspace{.15in}

\section{A model structure on $\itop$} \label{modelstr}

In this section we will describe a cofibrantly generated model structure on $\itop$. For the general theory and recognition principles of cofibrantly generated model categories, one may consult \cite{hovey}. For the category $\mathsf{Top}$, there is a cofibrantly generated model structure with generating cofibrations given by $I=\{S^{n-1} \to D^n\}$ and generating acyclic cofibrations given by $J=\{D^n \times \{0\} \to D^n \times [0,1]\}$. We start by stating Kan's transfer principle (\cite[11.3.2]{hirsch}).

\begin{prop}
	(modified transfer principle \cite[A.1]{stephanelm}) Let $\cC$ be a cofibrantly generated model category with generating cofibrations $I$ and generating acyclic cofibrations $J$. Let $\sD$ be a category with all small limits and colimits. Suppose $\{F_\iota: \cC \rightleftarrows \sD :U_\iota\}$ is a set of adjunctions. If $FI=\cup_{\iota}\{F_\iota(f);f \in I\}$ and  $FJ=\cup_\iota \{F_\iota(f);f\in J\}$ permit the small object argument, and for all $\iota$ the functor $U_\iota$ sends relative $FJ$-cell complexes to weak equivalences, then $\sD$ is a cofibrantly generated model category with generating cofibrations $FI$, generating acyclic cofibrations $FJ$, and weak equivalences and fibrations defined by $U_\iota$.
\end{prop}

We use this to transfer the cofibrantly generated model structure on $\mathsf{Top}$ to $\itop$. We note that $\sD$ is not required to be complete to transfer the model \emph{structure}.

\begin{thm} \label{modelstrisvt}
	There is a cofibrantly generated model structure on $\itop$ with generating cofibrations given by \[\{\Delta^{\mathbf{H}}_G \times S^{n-1} \to \Delta^{\mathbf{H}}_G \times D^n \}\] and generating acyclic cofibrations given by \[\{\Delta^{\mathbf{H}}_G \times D^n \times \{0\} \to \Delta^{\mathbf{H}}_G \times D^n \times [0,1]\},\] where $\mathbf{H}=H_0 < \cdots < H_k$ is an object of the link orbit category $\mathcal{L}_G$ and $n \in \NN$.
\end{thm}

We prove the theorem by defining the necessary adjunctions and checking the conditions of Kan's transfer principle.

\begin{prop} \label{adjunction}
	For each object $\mathbf{H}=H_0< \cdots < H_n$ of $\mathcal{L}_G$, there is an adjunction
	\[
\Delta_G^{\mathbf{H}} \times -: \xymatrix{\mathsf{Top} \ar@<1ex>[r] & \itop \ar@<1ex>[l] } : \imap(\Delta^{\mathbf{H}}_G, -).
	\]
\end{prop}

This is an easy consequence of using the subspace topology for $\imap(X,Y)$.

Let $U_{\mathbf{H}}=\imap(\Delta^\mathbf{H}_G,-)$, and let $FI=\bigcup_{\mathbf{H} \in \mathcal{L}_G}\{\Delta^\mathbf{H}_G\times f: f \in I\}$ and $FJ=\bigcup_{\mathbf{H} \in \mathcal{L}_G}\{\Delta^\mathbf{H}_G\times f: f \in J\}$ where $I$ and $J$ are the generating cofibration sets for $\mathsf{Top}$. 

\begin{lemma}
	The domains of $FI$ (resp, $FJ$) are small with respect to relative $FI$ (resp, $FJ$)-cell complexes, and $U_{\mathbf{H}}$ takes relative $FJ$-cell complexes to weak equivalences for all objects $\mathbf{H}$ of $\mathcal{L}_G$.
\end{lemma}

\begin{proof}
	The domains $\Delta^\mathbf{H}_G \times S^{n-1}$ and $\Delta^\mathbf{H}_G \times D^n$ are compact, and thus small with respect to relative cell complexes. 

	To show that $\imap(\Delta^\mathbf{H}_G,-)$ takes relative $FJ$-cell complexes to weak equivalences, 
	we first note that the map $\Delta^\mathbf{K}_G \times D^n \times [0,1] \to \Delta^\mathbf{K}_G \times D^n$ given by collapsing $[0,1]$ to $\{0\}$ yields an isovariant deformation retraction. 
 Pushouts preserve deformation retractions. Thus the maps in the sequential colimit defining a relative $FJ$-cell complex are isovariant homotopy equivalences.

 Given an isovariant homotopy $K:W \times [0,1] \to Z$, one can construct a homotopy $\imap(\Delta^\mathbf{H}_G,W) \times [0,1] \to \imap(\Delta^{\mathbf{H}},Z)$ which sends $(\lambda,t)$ to $K_t \circ \lambda$ for all $\mathbf{H}$. This implies that an isovariant homotopy equivalence $f:X \to Y$ has the property that $\imap(\Delta^\mathbf{H}_G,f)$ is a homotopy equivalence of spaces for all $\mathbf{H}$, and thus also a weak equivalence of spaces. Then the maps in the sequential colimit defining a relative $FJ$-cell complex are weak equivalences.

Finally, a map from $\Delta^\mathbf{H}_G$ to the sequential colimit $A \to \col_n X_n$ defining a cell complex will factor through a finite stage because the domain is compact and the quotients $X_n/A$ are $T_1$ (that is, points are closed). This holds because the maps $X_i \to X_{i+1}$ are pushouts along maps of $FJ$, which have the property that successive quotients are $T_1$.  
\end{proof}

Thus in $\itop$, 
a map $f:X \to Y$ is an isovariant weak equivalence (respectively, fibration) if $\imap(\Delta^\mathbf{H}_G,X) \to \imap(\Delta^\mathbf{H}_G,Y)$ is a weak equivalence (resp., fibration) for all $\mathbf{H} \in \mathcal{L}_G$.

\begin{ex}
	Consider the $C_2$-isovariant map $f:\ast \to D^2$ from Example \ref{exD}. This is an equivariant weak equivalence because the induced maps $f^e,f^G$ are both weak equivalences of spaces. But $f$ is not an isovariant weak equivalence because $\imap(\Delta^e_G,\ast)$ is empty while $\imap(\Delta^e_G, D^2)$ is not.
\end{ex}
 
\begin{rem} 
Since $\Delta^{H_0}_G \cong G/H_0$, the space $\imap(\Delta^{H_0}_G,X)$ is homeomorphic to $X_{H_0}$, the points of $X$ which are fixed by $H_0$ but no larger subgroup. That is, $\imap(\Delta^e_G,X)$ consists of free points $X \setminus \cup_{H \neq e}  X^H$, and $\imap(\Delta^G_G,X)=X^G$. The space $\imap(\Delta^{e <G}_G,X)$ is equivalent to a free neighborhood of the fixed points. 
\end{rem}

Finally, we check whether $\itop$ has all small limits and colimits. Colimits in $\itop$ can be constructed as the underlying colimit in $\etop$.

\begin{rem} \label{referee}
	The referee has pointed out that the corresponding category of compactly generated weak Hausdorff spaces with $G$-action would not have certain colimits. For example, consider the pushout of the following corner in $\itop$ with $G=C_2$, where $i$ is inclusion and the generator of $C_2$ acts on everything by multiplication by $-1$.
	\[
\xymatrix{(-\infty,0)\cup(0,\infty) \ar[r]^(.73)i \ar[d]_{\frac{x}{|x|}} & \mathbb{R} \\ \{\pm 1\} }
	\]
	The pushout in ordinary spaces is a three-point space which is not Hausdorff. After weak Hausdorffification, the pushout is a single point with trivial $C_2$-action, but there are no isovariant maps from any of the spaces in the diagram to a trivial point. 
\end{rem}

On the other hand, the underlying product $X \times Y$ is not a product in the isovariant category because if $Y$ has nontrivial $G$-action, the projection $X \times Y \to X$ is not necessarily isovariant. Thus we must delete subspaces of the underlying product to define the isovariant product.
\begin{defn}
	The isovariant product of $G$-spaces is the subspace of the product in $\topp$ comprised of tuples of points which all have the same isotropy group. We denote the isovariant product of two $G$-spaces by $X \overset{\mathsf{isvt}}{\times} Y$. 
\end{defn}

For example, if $X$ has a trivial $G$ action then $X \overset{\mathsf{isvt}}{\times} G$ is empty. If $X=[-1,1]$ with $C_2$-action given by negation, the isovariant product is 
\[X \overset{\mathsf{isvt}}{\times} X = [-1,1] \times [-1,1] \setminus \left( \{0\} \times [-1,1] \cup [-1,1] \times \{0\}\right) \, \cup \, (0,0).\] 
That is, the isovariant product $X \overset{\mathsf{isvt}}{\times} X$ is made up of points with either both coordinates zero or both nonzero. 

Thus $\itop$ also has all (nonempty) small limits, since it has all nonempty products, and all equalizers.  
To complete the category, we must assign it a formal terminal object $T$. We denote this new category $\itop^\term$. We show now that this is a cofibrantly generated \emph{model category}, that is, a category with all small limits and colimits with a cofibrantly generated model structure.

\begin{prop} \label{modelcatisvt}
	The category $\itop^\term$ is a cofibrantly generated model category with the same generating cofibrations and generating acyclic cofibrations as in Proposition \ref{modelstrisvt}. The new weak equivalences are $\mathcal{W}^\term$.
\end{prop}

\begin{proof}
	We will add some of the new maps $X \to T$ to the classes of fibrations and weak equivalences. A map $X \to T$ is in $\mathcal{W}^\term$ if $X$ is isovariantly contractible, that is, $\imap(\Delta^\mathbf{H},X)\simeq \ast$ for each chain of subgroups $\mathbf{H}$. We assume that the formal terminal object $T$ is isovariantly contractible. Thus a map $f:X \to T$ is a weak equivalence if $\imap(\Delta^\mathbf{H},f)$ is a weak equivalence of spaces for each $\mathbf{H}$.
	We will check the conditions of \cite{hovey}[2.1.19] to show that $\itop^\term$ is a cofibrantly generated model category, where $I=\{\Delta^\mathbf{H}_G \times S^{n-1} \to \Delta^\mathbf{H}\times D^n\}$ and $J=\{\Delta^\mathbf{H}_G \times D^n \to \Delta^\mathbf{H}\times D^n \times [0,1]\}$, as in Proposition \ref{modelstrisvt}. 

	For clarity, when refering to classes of maps in $\itop^\term$ with lifting properties against the generating (acyclic) cofibrations, we will use notation $I^\term$ (respectively, $J^\term$). For example $I$-inj $\subseteq I^\term$-inj, because there are more maps in $\itop^\term$ which have the right lifting property with respect to $I$ than in $\itop$.

	The domains of $I$ and $J$ have not changed, so are still small relative to $I$-cell and $J$-cell. The subcategory $\mathcal{W}^\term$ is closed under retracts, since the only new retracts of maps in $\mathcal{W}$ are of the form $Y \to T$, where $Y$ is a retract of $X$, so is also isovariantly contractible. The subcategory $\mathcal{W}$ is closed under 2-out-of-3, as the new compositions to consider are of the form $X \to Y \to T$ or $X \to T \to T$, and weak equivalences of spaces are closed under 2-out-of-3.

	To show that $I^\term$-inj$=\mathcal{W}^\term \cap J^\term$-inj, we use the adjunction of Proposition \ref{adjunction}. The map $f$ in $\itop^\term$ is in $I^\term$-inj if and only if $f$ has the right lifting property with respect to maps in $I$. By adjunction, for each chain of subgroups $\mathbf{H}$, $\imap(\Delta^\mathbf{H},f)$ has the right lifting property with respect to maps in $I^\topp=\{S^{n-1} \to D^{n}\}$. Because $\topp$ is a cofibrantly generated model category, this holds if and only if $\imap(\Delta^\mathbf{H},f) \in \mathcal{W}^\topp \cap J^\topp$-inj. By definition, $\imap(\Delta^\mathbf{H},f) \in \mathcal{W}^\topp$ for all $\mathbf{H}$ if and only if $f \in \mathcal{W}^\term$, and $\imap(\Delta^\mathbf{H},f) \in J^\topp$-inj if and only if $f \in J^\term$-inj.

	Finally, $J$-cell $\subseteq \mathcal{W} \cap I$-cof, because there are no new elements of $J$-cell in $\itop^\term$ that were not in $\itop$. Since $ \mathcal{W} \subseteq \mathcal{W}^\term$ and $I$-cof $\subseteq I^\term$-cof, this completes the proof.
\end{proof}

We note that the new fibrations of $\itop^\term$ are $X \to T$ which have the right lifting property against the generating acyclic cofibrations and the new acyclic fibrations are $X \to T$ which lift against the generating cofibrations. 
Since $T$ is formal, the property of being a fibrant object of $\itop^\term$ is the same as possessing an extension property along the generating acyclic cofibrations.

\vspace{.15in}

\section{Isovariant Elmendorf's theorem} \label{elmsection}

In this section, we show that for a finite group $G$, the category $\itop^\term$ of $G$-spaces with isovariant maps given the model structure of Proposition \ref{modelcatisvt} is Quillen equivalent to a category of functors on the link orbit category. This is analogous to Elmendorf's theorem for equivariant spaces, where the diagrams are on the orbit category $\mathcal{O}_G$ \cite{elmthm, alaskanotes}. The argument is mostly formal, except for the proof that the links $\imap(\Delta^\mathbf{H}_G,-)$ preserve certain homotopy colimits up to homotopy.

We will denote by $\fun(\mathcal{L}_G^{op},\mathsf{Top})$ the category whose objects are functors $F:\mathcal{L}_G^{op} \to \mathsf{Top}$ with morphisms given by natural transformations.
Because $\mathcal{L}_G^{op}$ is small, the cofibrantly generated model structure on $\mathsf{Top}$ can be transferred across an adjunction to endow $\fun(\mathcal{L}_G^{op},\mathsf{Top})$ with the projective model structure, where the weak equivalences and fibrations are levelwise, and generating cofibrations are given by $\{ \Hom_{\mathcal{L}_G}(-,\mathbf{H}) \times I\}$, where $I$ denotes the collection of generating cofibrations of $\topp$ \cite{hirsch}. More details about the projective model structure on diagrams can be found in \cite{piacenza} or \cite[11.6.1]{hirsch}.

\begin{thm} \label{mainthm}
	Let $G$ be a finite group. The following adjunction is a Quillen equivalence.
	\[
\Delta^\bullet_G \otimes_{\mathcal{L}_G} -: \xymatrix{ \mathsf{Fun}(\mathcal{L}^{op}_G, \mathsf{Top}) \ar@<1ex>[r] & \itop^\term \ar@<1ex>[l]}: \imap(\Delta^\bullet_G,-).
	\]
\end{thm}

The proof is given in steps. We start with a key technical lemma showing that the functors $\imap(\Delta^\mathbf{H}_G,-)$ are cellular. 

\begin{lemma} \label{linkscellular}
	The functors $\imap(\Delta^{\mathbf{H}}_G,-)$ are homotopically cellular in the sense that they preserve homotopy pushouts and sequential homotopy colimits up to homotopy.
\end{lemma}

\begin{proof}
	We need only consider the isovariant mapping spaces from the fundamental domains of the linking simplices $\Delta^{\mathbf{H}}_G$, which are simplices $\Delta^k$ with specified isotropy groups on faces. Further, as long as we do not collapse the dimension of the simplex, the $G$-action is not necessary for the proof. In particular, we may use the stratification on $\Delta^k$ given by the isotropy groups of a fundamental domain of the linking simplex. Let $k \geq 0$ and let $L$ denote the functor $\map_{\strat}(\Delta^k,-):\strat\topp \to \topp$ of continuous stratified maps. 
	Let $\cX$ denote the pushout diagram $B \xleftarrow{f} A \xrightarrow{g} C$, and consider the model for the homotopy pushout given by $\hc \cX=(B\times \{0\}) \coprod_f (A \times [0,1]) \coprod_g (C \times \{1\})$. Define $p_c: \hc \cX \to [0,1]$ to be projection to the cylinder coordinate, with $B$ sent to 0 and $C$ sent to 1. 

	The universal property of homotopy pushouts gives rise to the map $\phi:\hc L(\cX) \to L(\hc \cX)$, which we will show is a weak equivalence. This weak equivalence then restricts to a weak equivalence on the homotopy colimits of mapping spaces on the fundamental domains of the linking simplices, which can then be extended to the full linking simplices. The elements of $\hc L(\cX)$ are maps of simplices which are constant in the cylinder coordinate, or after application of $p_c$. 
	Define a deformation retraction $cr:\hc \cX \to \hc \cX$ which collapses $A \times [0,1/4]$ to $B$ by $f$ and $A \times [3/4, 1]$ to $C$ along $g$ and stretches $A \times (1/4,3/4)$ to $A \times (0,1)$. The choice of $1/4$ is for convenience; any number less than $1/2$ would suffice. Then there is an equivalence $\hc L(\cX) \xrightarrow{\simeq} cr \hc L (\cX)$. 
	Let $n \geq 0$ and consider the commuting diagram for the pairs $(D^n,S^{n-1}) \to (L(\hc \cX), \hc L(\cX))$. We will show $\psi$ lifts up to homotopy by defining a map $D^n \to cr \hc L(\cX)$ which makes the top and bottom triangles commute up to compatible homotopy.
	\[
\xymatrix{ S^{n-1} \ar[r] \ar[d] & \hc L(\cX) \ar[d]_(.35)\phi \ar[r]^{\simeq} & cr \hc L(\cX)  \ar[d] \\
D^n \ar[r]_(.3){\psi} \ar@{-->}[urr] & L(\hc \cX) \ar[r]^\simeq & L(cr \hc \cX)}
	\]

	For a disk of links in the homotopy colimit, that is, $\psi(D^n)$, we will define a continuous truncation of the links that we can then ``flatten out'' in the cylinder direction of the homotopy colimit. 

	To define the truncation, we will define a lower semicontinuous function 
	\[L(\hc \cX) \to (0,1]\] 
	which takes a simplex in the homotopy pushout and returns a real number measuring the shortest distance (in the simplex) from the vertex $v_0=(1,0, \dots, 0)$ to points of the simplex which travel more than 1/4 in the cylinder coordinate of the homotopy pushout. For a simplex of dimension $k=1$, the output measures the first time the path travels more than 1/4 in the cylinder coordinate from its starting point. First, we normalize the starting point (or vertex $v_0$) of all links to be at 0 in $\mathbb{R}$. Let $\iota: [0,1] \to \mathbb{R}$ be inclusion, and let $L_0(\mathbb{R})$ be the space of links in $\mathbb{R}$ that start at 0; that is, $L_0(\mathbb{R})=\{\map(\Delta^k, \mathbb{R}): v_0 \mapsto 0\}$. Let $\gamma \in L(\hc \cX)$, so $\gamma:\Delta^k \to \hc \cX$. Then $\widetilde{\gamma} \in L_0(\mathbb{R})$ is defined by $(t_0, \dots, t_k) \mapsto \iota p_c \gamma(t_0, \dots, t_k) - \iota p_c \gamma(v_0)$. The normalization map $L(\hc \cX) \to L_0(\mathbb{R}): \gamma \mapsto \widetilde{\gamma}$ is continuous. 

	Then define $\ell:L_0(\mathbb{R}) \to (0,1]$ by
	\[\ell(\widetilde{\gamma})=\inf \left\{ \sum_{i=1}^k t_i: | \widetilde{\gamma}(t_0, \dots, t_k)| \geq 1/4\right\}.\]

	Let $0<\alpha <1$. The function $\ell$ is lower semicontinuous if all the sets $U_\alpha=\{ \gamma \in L_0(\mathbb{R}): \ell(\gamma)> \alpha \}$ are open using the compact open topology on $L_0(\mathbb{R})$. Let $\gamma \in U_\alpha$, so $\ell(\gamma)=\alpha+\epsilon$ for some $\epsilon>0$. Define the open set $W \subseteq L_0(\mathbb{R})$ to be the set of functions which take the compact set $\{(t_0, \dots, t_k) \in \Delta^k: \sum_{i=1}^k t_i \leq \alpha + \epsilon/2 \}$ to the open set $(-1/4, 1/4) \subset \mathbb{R}$. We see $\gamma \in W$, since $\ell(\gamma)=\alpha + \epsilon$. For example, in the $k=1$ case, we can think of $\gamma:\Delta^1 \to \mathbb{R}$ as a function of the variable $t_1$. Then $\gamma(\{(t_0,t_1): 0 \leq t_1 \leq \alpha + \epsilon/2 \})$ must fall within the open set $(-1/4,1/4)$, since $\gamma$ takes the point with $t_1=\alpha+\epsilon$ to either 1/4 or -1/4, and no smaller value of $t_1$ leaves the set $(-1/4,1/4)$ by definition of infimum. We also see that $W \subseteq U_\alpha$ since all simplices $\gamma \in W$ have the property that $\ell(\gamma)> \alpha+\epsilon/2>\alpha$. Thus we have shown that $U_\alpha$ is open and $\ell$ is lower semicontinuous.

	The composition of $\ell$ with the normalization $L(\hc \cX) \to L_0(\mathbb{R}) \to (0,1]$ is also lower semicontinuous. The image $\psi(D^n)$ in $L(\hc \cX)$ is a compact subset, so the composition $D^n \to (0,1]$ is lower semicontinuous and achieves its minimum value, $\ell_{min}$. Then there is a \emph{continuous} function $D^n \to (0,1]$ which takes the constant value $\ell_{min}$ on the entire disk.

	Thus we have a continuous truncation for a disk of links. That is, for $\gamma$ in the image $\psi(D^n)$, let $\gamma_{\ell}: \Delta^k \to \hc \cX$ denote precomposition of $\gamma$ with the map which linearly scales $\Delta^k$ to $\{(t_0,\dots,t_k) \in \Delta^k: \sum_{i=1}^k t_i \leq \ell_{min} \}$. (Note this scaling is isovariant on $\Delta^\mathbf{H}_G$.) The simplex $\gamma_\ell$ is a truncated version of $\gamma$ which stretches no more than 1/4 in the cylinder coordinate. 

	The last step of the map $D^n \to cr \hc L(\cX)$ is to flatten the truncated links to be constant in the cylinder direction. We do this piecewise by postcomposing with projection. If $p_c \gamma(v_0) \in (1/4,3/4)$, project $\gamma_\ell$ (which remains completely in the cylinder) along the cylinder coordinate to the element of $\hc L(\cX)$ with $p_c \gamma(v_0)$ in the cylinder coordinate. (This can be done isovariantly, since the cylinder of the homotopy pushout has trivial group action.) If $p_c \gamma(v_0) \in [0,1/4]$, apply the map $f:A \to B$ to $\gamma_\ell$ to push the link into $B$. Similarly, if $p_c \gamma(v_0) \in [3/4,1]$, use $g:A \to C$ to push $\gamma_\ell$ to $C$. Since the truncation is defined continuously on the disk of links $\psi(D^n)$, this defines a continuous function from the image of the disk $D^n$ in $L(\hc \cX)$ to $cr \hc L(\cX)$.

	The image of the map $S^{n-1} \to \hc L(\cX)$ must contain links which are constant in the cylinder coordinate, and these remain constant after collapsing the first and last quarter of the cylinder using $cr$. Thus the homotopy of the top triangle involves the (isovariant) deformation retraction between the links and their truncations $\gamma_\ell$. The bottom triangle commutes up to homotopy using the same deformation, but also using the straight line homotopy through the cylinder coordinate to unflatten the links. 

	Thus the map $\phi: \hc L(\cX) \to L(\hc \cX)$ is a weak equivalence of spaces and $\imap(\Delta^\mathbf{H}_G,-)$ commutes with homotopy pushouts. The proof for directed homotopy colimits proceeds similarly, using the mapping telescope and collapsing only $[3/4,1]$ in each interval. 
\end{proof}

\begin{prop}
	The maps in Theorem \ref{mainthm} form a Quillen adjunction.
\end{prop}

\begin{proof}
	This is formal using the projective model structure on the category of $\mathcal{L}_G^{op}$-diagrams. In a Quillen adjunction, the right adjoint $\imap(\Delta_G^\bullet,-)$ preserves fibrations and acyclic fibrations. A fibration (resp., weak equivalence) $f$ in $\itop^\term$ is a map for which $\imap(\Delta_G^\mathbf{H},f)$ is a fibration (resp., weak equivalence) in $\mathsf{Top}$ for all $\mathbf{H}$. This defines a natural transformation which is levelwise a fibration (resp., weak equivalence). 
\end{proof}

\begin{prop}
	For cofibrant $F \in \mathsf{Fun}(\mathcal{L}^{op}_G, \mathsf{Top})$, the unit $\eta:F \to \imap(\Delta^\bullet_G, \Delta^\bullet_G \otimes_{\mathcal{L}_G} F)$ is a weak equivalence.
\end{prop}

\begin{proof}
To show that the unit $\eta$ is a weak equivalence of diagrams requires showing that it is a levelwise weak equivalence of spaces. If $F$ is a cofibrant $\mathcal{L}^{op}_G$-diagram, then $F$ is a retract of an $\mathcal{L}_G(-,\mathbf{H})\otimes I$-cell complex \cite{piacenza}. Since weak equivalences are closed under retracts, we need only show that $\eta$ is an equivalence for cell complexes. By Lemma \ref{linkscellular}, the functors $\imap(\Delta^{\mathbf{H}}_G,-)$ preserve homotopy pushouts and sequential homotopy colimits, so it is enough to show $\eta$ is an equivalence on the representable functors $\mathcal{L}_G(-,\mathbf{H})$.

Formally, $\Delta^\bullet_G \otimes_{\mathcal{L}_G} \mathcal{L}_G(\bullet,\mathbf{H}) \cong \Delta^{\mathbf{H}}_G$, using the following adjunctions for a $G$-space $X$ and the Yoneda lemma. 
\[
\Hom_\mathsf{isvt}(\Delta^\bullet_G \otimes_{\mathcal{L}_G} \mathcal{L}_G(\bullet,\mathbf{H}), X) \cong \Hom_{\mathsf{Fun}}(\mathcal{L}_G(\bullet,\mathbf{H}), \imap(\Delta^\bullet_G, X))\cong \Hom_\mathsf{isvt} (\Delta^\mathbf{H}_G, X)
\]

Then the desired levelwise weak equivalence boils down to showing
\[
\mathcal{L}_G(\mathbf{K},\mathbf{H}) \simeq \imap(\Delta^\mathbf{K}_G, \Delta^\mathbf{H}_G).
\]
The morphisms in $\mathcal{L}_G$ have the discrete topology, so we will show that $\imap(\Delta^\mathbf{K}_G, \Delta^\mathbf{H}_G)$ is a disjoint union of contractible components. 

Let $f:\Delta^\mathbf{K} \to \Delta^\mathbf{H}$ be an isovariant map. Then $\langle g,x \rangle =\langle g,(t_0, \dots, t_{n-i},0 \dots, 0)\rangle  \in \Delta^\mathbf{K}$ is fixed by $gK_ig^{-1}$ under the $G$-action, so by isovariance $f(\langle g,x\rangle )$ is also fixed by $gK_ig^{-1}$ and there is a $j_i$ such that $H_{j_i}$ is conjugate to $K_i$. Thus the subgroups of the chain $\mathbf{K}$ are conjugate to subgroups in $\mathbf{H}$. 
Further, $f$ corresponds to an element $\gamma \in G$ which simultaneously conjugates $K_i$ to $H_{j_i}$. This can be seen by using an element $\gamma \in G$ which conjugates $K_0$ to $H_{j_0}$. By continuity of $f$, the same $\gamma \in G$ will work for all boundaries of the simplex, thus for all $K_i$.
That is, every element of $\imap(\Delta^\mathbf{K}, \Delta^\mathbf{H})$ corresponds to an element $(\iota, \overline{\gamma})$ of $\mathcal{L}_G(\mathbf{K},\mathbf{H})$. 

Two distinct elements $(\iota,\overline{\gamma}), (\iota',\overline{\gamma'})$ of $\mathcal{L}_G(\mathbf{K},\mathbf{H})$ produce two maps $\Delta^{(\iota, \overline{\gamma})}, \Delta^{(\iota',\overline{\gamma'})}$ in different path components of $\imap(\Delta^\mathbf{K}, \Delta^\mathbf{H})$ because there is no \emph{isovariant} homotopy between them. 

By convexity of $\Delta^n$, any isovariant map $\Delta^\mathbf{K} \to \Delta^\mathbf{H}$ corresponding to $(\iota, \overline{\gamma})$ is isovariantly homotopic to $\Delta^{(\iota,\overline{\gamma})}$ through a straight-line homotopy. Thus $\pi_0(\imap(\Delta^\mathbf{K},\Delta^\mathbf{H}))\cong \mathcal{L}_G(\mathbf{K},\mathbf{H})$. Since the identity on the $(\iota, \overline{\gamma})$ component of $\imap(\Delta^\mathbf{K}, \Delta^\mathbf{H})$ is homotopic to the constant map $\Delta^{(\iota,\overline{\gamma})}$, each path component of $\imap(\Delta^\mathbf{K}, \Delta^\mathbf{H})$ is contractible.
\end{proof}

\begin{proof}[Proof of Theorem \ref{mainthm}]
The adjunction of Theorem \ref{mainthm} is a Quillen equivalence if the following condition holds: $\Delta^\bullet_G \otimes F \to Y$ is a weak equivalence in $\itop^\term$ if and only if $F \to \imap(\Delta^\bullet_G, Y)$ is a weak equivalence of diagrams $\mathsf{Fun}(\mathcal{L}_G^{op},\mathsf{Top})$, where the diagram $F$ is cofibrant and the isovariant space $Y$ is fibrant.

The map $F \to \imap(\Delta^\bullet_G, Y)$ factors as $F \xrightarrow{\simeq} \imap(\Delta^\bullet_G, \Delta^\bullet_G \otimes F) \to \imap(\Delta^\bullet_G, Y)$. The 2-out-of-3 property for weak equivalences completes the proof.
\end{proof}








\bibliography{mybib}{}
\bibliographystyle{amsalpha}

\end{document}